\documentclass[a4paper,11pt]{article}
\usepackage{mathrsfs}
\usepackage{latexsym}
\usepackage{amsthm}
\usepackage{amsfonts}
\usepackage[usenames]{color}
\usepackage{graphicx}
\usepackage{amsthm}
\usepackage{mathrsfs}
\usepackage{dsfont}
\usepackage{indentfirst}
\usepackage{enumerate}
\usepackage[square, comma, sort&compress, numbers]{natbib}
\usepackage{amsmath,amssymb,bm,esdiff}
\usepackage[toc,page,title]{appendix}
\usepackage{etoolbox}
\usepackage{geometry}
\usepackage{type1cm}
\usepackage{tikz}
\usepackage{float}
\usetikzlibrary{shapes.geometric, arrows.meta, positioning, calc, patterns, intersections}
\newtheoremstyle{mythm}{1.5ex plus 1ex minus .2ex}{1.5ex plus 1ex
minus .2ex}{\kai}{\parindent}{\song\bfseries}{}{1em}{}
\numberwithin{equation}{section}

\newtheorem{thm}{Theorem}[section]
\newtheorem{lemma}{Lemma}[section]
\newtheorem{proposition}{Proposition}[section]
\newtheorem{remark}{Remark} [section]
\newtheorem{corollary}{Corollary}[section]
\allowdisplaybreaks[4]

\title{Regularity of Free Boundary for Monge-Amp\`{e}re Obstacle Problem up to the Boundary}
\author{Hanfei Han and Genggeng Huang}
\date{\vspace{-2em}}
\begin{document}
\maketitle
\abstract{}
In this paper, we study the regularity of free boundaries for Monge-Amp\`{e}re obstacle problem up to the boundary. We mainly focus on how the growth order of boundary value near the contact point impact the regularity. We also find an interesting phenomena that under quadratic growth boundary value, the free boundary enjoys better regularity in high dimensions than in two dimension.
\section{Introduction}
\indent In this paper we study the regularity of free boundaries for Monge-Amp\`{e}re obstacle problem,
\begin{equation}
\left\{
\begin{aligned}
\det D^2v&=\chi_{\{v>0\}}\quad\mathrm{in}\quad \Omega\\
v&=v_0\qquad\quad\mathrm{on}\quad\partial\Omega.
\end{aligned}
\right.
\label{1.1}
\end{equation}
Here $\Omega\subseteq\{x_n<0\}$ is a bounded convex domain with a flat boundary part $S$ such that
\begin{equation}
0\in S\subseteq\{x_n=0\}.
\label{1.2}
\end{equation}
$v_0$ is a non-negative convex function in $C^{1,\beta}(\bar{\Omega})$ for some $\beta\in(1-\frac{2}{n},1)$,
\begin{equation}
v_0(0)=0,\quad v_0>0\quad\mathrm{on}\quad\partial\Omega\backslash\{0\}.
\label{1.3}
\end{equation}
Here, the free boundary refers to $\Gamma:=\partial\{v=0\}$, where $\{v=0\}=\{x\in\bar{\Omega}|v(x)=0\}$.
\par Then the existence of solutions to \eqref{1.1} follows from \cite{2}. And by the argument in \cite{6}, the set $\{v=0\}$ has a non-empty interior or $\{v=0\}=\{0\}$. Since in the latter case the free boundary degenerates to a point $0$, we only consider the former case(See Figure \ref{1}). We can also show by the argument in \cite{6} that $v$ is locally strictly  convex in $\Omega\bigcap\{v>0\}$ and hence $v\in C^\infty(\Omega\bigcap\{v>0\})$ (see \cite{7} for instance).
\begin{figure}[H]
\centering
\caption{}
\begin{tikzpicture}[scale=0.4]
\draw[->] (-7,0) -- (7,0) node[right] {$x'$};
\draw[->] (0,-8) -- (0,1) node[above] {$x_n$};
\draw[thick]
(-2,0) -- (5,0) 
.. controls (8,-4) and (4,-6) .. (0,-5.7) 
.. controls (-6,-5) and (-5,-4) .. (-2,0); 
\fill[pattern=north east lines] 
      (0,0) .. controls (1.8,0) and (1.5,-1.3) .. (0,-1.2) 
       .. controls (-1.0,-1.0) and (-1.2,0) .. (0,0);
\draw 
(0,0) .. controls (1.8,0) and (1.5,-1.3) .. (0,-1.2) 
.. controls (-1.0,-1.0) and (-1.2,0) .. (0,0);
\node at (0.5,-2.5) {\Huge$\Omega$}; 
\node at (-0.8,0.45) {\Large$S$};
\node (Ulabel) at (3,-0.5) {$\{v=0\}$};
\draw[->] (Ulabel) -- (0.6,-0.6);
\node (rylabel) at (2,1.1) {$\Gamma$};
\draw[->] (rylabel) -- (0.78, -0.14);  
\end{tikzpicture}
\label{1}
\end{figure}
This type of Monge-Amp\`{e}re obstacle problem was initially studied by Savin in \cite{2}. More precisely, he studied
\begin{equation*}
\left\{
\begin{aligned}
\det D^2v&=f\chi_{\{v>0\}}\quad\mathrm{in}\quad\Omega\\
v&=1\qquad\quad\mathrm{on}\quad\partial\Omega
\end{aligned}
\right.
\end{equation*}
for positive functions $f$ on $\bar{\Omega}$, where the free boundary locates strictly inside $\Omega$. He proved that the free boundary is strictly convex and $C^{1,\alpha}$ provided $f$ is bounded away from $0$ and $+\infty$, and $C^{1,1}$ provided $f=1$. Recently the higher order regularity for free boundary was studied in \cite{5} and \cite{1}. The free boundary is smooth and analytic, if $f$ is respectively smooth and analytic, which was established for dimension 2 in \cite{5} by G\'{a}lvez, Jim\'{e}nez, and Mira and for all dimensions in \cite{1} by Huang, Tang and Wang later. A related interesting model is\\
\begin{equation*}
\det D^2v=gv^q\chi_{\{v>0\}},\qquad v\geq0\qquad\mathrm{in}\quad\Omega
\end{equation*}
with constant $q\in[0,n)$ and positive function $g$ on $\Omega$. This was initially studied by Daskalopoulos and Lee for dimension 2 in \cite{9}. They proved the free boundary is smooth provided $v>0$ on $\partial\Omega$ and $g=1$. Later, in \cite{3}, Jin, Tu and Xiong proved that for all dimensions, the strictly convex part of free boundary inside $\Omega$ is $C^{1,\gamma}$ and $C^{2,\gamma}$ for some constant $\gamma\in(0,1)$, if $g$ is respectively bounded away from $0$ and $+\infty$ and H\"{o}lder continuous. As a complement, we mention that \cite{3} also gives the complete classification result on the global solution of Monge-Amp\`{e}re obstacle problem
\begin{equation*}
\det D^2v=\chi_{\{v>0\}},\quad v\geq0\quad\mathrm{in}\quad\mathbb{R}^n.
\end{equation*}
However, little has been known about the regularity of free boundary when it touches the fixed boundary $\partial\Omega$ so far.
\par To study this, we consider a natural model with the obstacle problem \eqref{1.1} and assumption \eqref{1.3} as $v_0$ having only one zero point on $\partial\Omega$, say, $0$. Then we can mainly focus on how the free boundary $\Gamma$ meets the fixed boundary $\partial\Omega$ at the contact point $0$. In this paper, we assume that $v_0$ has a fixed growth order of at least $2$ near $0$ on $S$, i.e.
\begin{equation*}
v_0(x)\sim|x|^{2+\frac{2}{n}(\alpha-1)},\quad x\in\partial\Omega\quad \mathrm{for}\quad\mathrm{some}\quad\alpha\ge 1.
\end{equation*}
Throughout the present paper, we use the following notation. For two quantities $a$, $b$, we denote $a\lesssim b$ if there is a constant $C$ such that $a\leq Cb$;  $a\gtrsim b$ if there is a positive constant $C$ such that $a\geq Cb$; and $a\sim b$ if  both $a\lesssim b$, $a\gtrsim b$ hold.

We mention here that the flatness assumption of part $S$ of the fixed boundary can be weakened to some extent while we do not deal with it in our main results to omit the trivial details (see Remark \ref{Remark 1.3.}).
\par The most technical part is that the boundary value $v_0(x)\sim|x|^2$ for $x\in\partial\Omega$. The first thing is that the asymptotic behavior of the free boundary up to the fixed boundary is exactly 
$$-x_n=h(x')\sim |x'|(-\ln |x'|)^{-n/2}.$$ The proof of this asymptotic behavior needs delicate construction of barrier functions. Due to this asymptotic behavior, the scaling argument fails.   We need to do some careful calculation to control the second derivative of the free boundary near $0$. A surprising phenomenon is discovered that the free boundary up to $0$ enjoys better regularity in $n\ge 3$ than in $n=2$. This is quite different from classical theory in elliptic partial differential equations as one usually expects better regularity in lower dimension than in higher dimension under the same conditions.  
\begin{thm}
\label{Theorem 1.1.} 
Let $v,v_0$ be given by \eqref{1.1}-\eqref{1.3}. In addition,  $v_0(x)\sim|x|^2$ on $\partial\Omega$,  there holds 
\begin{itemize}  
\item[(i)] $n=2$, then $\Gamma\in C^1$ and  $\Gamma\notin C^{1,Dini}$;
 \item[(ii)] $n\geq3$, then $\Gamma\in C^{1,Dini}$.
\end{itemize}
\end{thm}
We also have the following higher regularity results if the boundary value $v_0$ has higher order growth behavior at $0$.
\begin{thm}
\label{Theorem 1.2.}
Let $v,v_0$ be given by \eqref{1.1}-\eqref{1.3}. In addition,  $v_0(x)\sim |x|^{2+\frac{2}{n}(\alpha-1)}$ on $\partial\Omega$, $\alpha>1$, there holds
\begin{itemize}
    \item[(i)] $\alpha\in(1,2)$, then $\Gamma\in C^{1,\alpha-1}$ and $\Gamma\notin C^{1,1}$;
    \item[(ii)] $\alpha=2$, then $\Gamma\in C^{1,1}$;
    \item[(iii)] $\alpha>2$, then $\Gamma\in C^2$.
\end{itemize}
\end{thm}
The followings are some remarks on the main theorems.
\begin{remark}
\label{Remark 1.3.} Flatness of $S$ could be replaced by $S\in  C^2$ in Theorem \ref{Theorem 1.1.}. More generally, in Theorem \ref{Theorem 1.2.}, we can also replace the flatness of $S$ by the condition that $S$ is $k-$th tangent to the plane $x_n=0$ for some $k$ which is determined by $\alpha$. Since the proof is essentially the same  but more involved, we will omit it here. 
\end{remark}
\begin{remark}
\label{Remark 1.6}
For any smooth function $f$ bounded away from $0$ and $+\infty$, replacing $\det D^2 v=\chi_{\{v>0\}}$ by $\det D^2v=f\chi_{\{v>0\}}$ in \eqref{1.1}, both Theorem \ref{Theorem 1.1.} and Theorem \ref{Theorem 1.2.} would still hold.
\end{remark}
\begin{remark}
\label{Remark 1.4.}
Consider the classical obstacle problem
\begin{equation*}
\left\{
\begin{aligned}
\Delta v&=\chi_{\{v>0\}}\quad\mathrm{in}\quad\Omega\\
v&=\tilde{v}_0\qquad\mathrm{on}\quad\partial\Omega. 
\end{aligned}
\right.
\end{equation*}
Following the example of \cite{8,ShahgholianUraltsevaWeiss2007}, for all dimensions, one cannot expect the $C^{1,Dini}$ regularity of free boundary even we impose non-negative smooth data $\tilde{v}_0$ with infinite vanishing order.  The convexity of the solution $v$ and the set $\{v=0\}$ in Monge-Amp\`{e}re obstacle problem may play an important role in the improvement of regularity of the free boundary up to the boundary. 
\end{remark}
\begin{remark} In \cite{ChenLiuWang2023arxiv}, Chen-Liu-Wang studied the global regularity of free boundary $\Gamma$ up to the boundary for the partial optimal transport problem. This problem is quite different from our problem. The free boundary $\Gamma$ touches the fixed boundary transversally and enjoys optimal $C^{1,1-\epsilon}$ regularity, $\forall \epsilon\in (0,1)$. 
\end{remark}
Our paper is organized as follows. In Section 2, we introduce some preliminary results on Monge-Amp\`{e}re equation and the corresponding obstacle problem. In Section 3, we prove the $C^1$ regularity of free boundary. Then the case $v_0(x)\sim|x|^2$, $x\in\partial\Omega$ and the case $v_0(x)\sim|x|^{2+\frac{2}{n}(\alpha-1)}$, $x\in\partial\Omega$ for some $\alpha>1$ will be dealt with in Section 4 and Section 5 respectively.
\section{Preliminaries}
In this section, we state some preliminaries. We adapt some more general settings here. Suppose $\Omega$ is a bounded convex domain in $\mathbb{R}^n(n\geq2)$, $v_0$ be a non-negative function on $\partial\Omega$, and $\{v=0\}$ has non-empty interior, where $v$ solves \eqref{1.1}. The constant $C$ would change line by line throughout the paper unless otherwise specified.
\begin{lemma}
\label{Lemma 2.1}
(\cite[Theorem 8.3.2]{4}) Suppose $\nu$ is a finite measure in $\Omega$, $\alpha\in(0,1)$, $\varphi\in C^\alpha(\bar{\Omega})$ is a convex function, and $u$ solves
\begin{equation*}
\left\{
\begin{aligned}
\det D^2u=\nu\quad&in\quad \Omega\\
u=\varphi\quad&on\quad\partial\Omega. 
\end{aligned}
\right.
\end{equation*}
Then $u\in C^{\bar{\alpha}}(\bar{\Omega})$ for $\bar{\alpha}=\min\{\alpha,\frac{1}{n}\}$, and
\begin{equation*}
[u]_{C^{\bar{\alpha}}(\bar{\Omega})}\leq C([\nu(\Omega)]^\frac{1}{n}+[\varphi]_{C^\alpha(\bar{\Omega})}),
\end{equation*}
where $C$ is a positive constant depending on $n$ and $diam(\Omega)$.
\end{lemma}
\begin{lemma}
\label{Lemma 2.2}
Suppose the non-negative convex functions $\underline{w}$ and $\overline{w}\in C(\bar{\Omega})$ satisfy
\begin{equation*}
\det D^2\underline{w}\geq\chi_{\{\underline{w}>0\}}\quad\text{in}\quad\Omega,\quad \det D^2\overline{w}\leq\chi_{\{\overline{w}>0\}}\quad\text{in}\quad\Omega,
\end{equation*}
and
\begin{equation*}
\underline{w}\leq\overline{w}\quad on\quad\partial\Omega,
\end{equation*}
then $\underline{w}\leq\overline{w}$ in $\bar{\Omega}$.
\end{lemma}
\begin{lemma}
\label{Lemma 2.3}
Suppose $v$ is a solution of \eqref{1.1}, and $x_0\in\Omega$ is an exposed point of $\{v=0\}$, then the free boundary $\partial\{v=0\}$ is $C^{2,\gamma}$ in a small neighborhood of $x_0$ for some $\gamma\in(0,1)$.
\end{lemma}
\begin{lemma}
\label{Lemma 2.4}
Suppose $v$ is a solution of \eqref{1.1}, and $x_0\in\Omega$ is an extreme point of $\{v=0\}$, then $x_0$ is an exposed point of $\{v=0\}$.
\end{lemma}
We refer to \cite[Section 1-2]{3} for the proof of Lemma \ref{Lemma 2.2}, Lemma \ref{Lemma 2.3} and Lemma \ref{Lemma 2.4}.
\begin{corollary}
\label{Corollary 2.5}
Suppose $v$ is a solution of \eqref{1.1}, and $[x_1,x_2]$ is a segment on $(\partial\{v=0\})\bigcap\Omega$, then there exists $x_0\in\partial\Omega$ such that $[x_0,x_1]\subseteq\partial\{v=0\}$.
\end{corollary}

\begin{proof}
Let $H$ be a closed half space such that $$\{v=0\}\subset H,\quad [x_1,x_2]\subseteq \partial H.$$
Then, there exists $x_0\in \partial H$ such that $x_0$ is an extreme point of $\{v=0\}\bigcap \partial H$ relative in $\partial H$.
Also, $x_0$ is an extreme point of $\{v=0\}$ in $\mathbb{R}^n$.
\par If $x_0\in\Omega$, then Lemma 2.4 implies that $x_0$ is an exposed point of $\{v=0\}$, and it follows that $\partial\{v=0\}$ is $C^1$ near $x_0$ from Lemma 2.3.
Therefore, we can choose another closed half space $\tilde H$ such that $\{v=0\}\subset\tilde{H}$ and $\{v=0\}\bigcap\partial\tilde{H}=\{x_0\}$.
This implies $\{v=0\}$ has two different supporting planes $H,\tilde H$ at $x_0$ which contradicts $v\in C^1$ near $x_0$. This ends the proof of  Corollary   \ref{Corollary 2.5}.
\end{proof}
\begin{lemma}
\label{Lemma 2.6}
Suppose $v$ is a solution of \eqref{1.1}, $x_0\in\Omega$, $v(x_0)=0$, and $[x_0,y]$ is not contained in $\partial\{v=0\}$ for any $y\in\partial\Omega$, then $\partial v(x_0)=\{0\}$.
\end{lemma}

The proof of Lemma \ref{Lemma 2.6} could be completed with the same method as in \cite[Proposition 2.8]{2}.
\section{$C^1$ Regularity of Free Boundary}
This section is devoted to show the $C^1$ regularity of free boundary.  We will first show the strict convexity of $\{v=0\}$ in Lemma \ref{Lemma 3.1}.
\begin{lemma}
\label{Lemma 3.1}
(Strict convexity of free boundary) Suppose $\Omega$ is a bounded convex domain in $\mathbb{R}^n$ and $v_0$ is a non-negative continuous function on $\partial\Omega$  with isolated zeros $\{x_0,\cdots,x_N\}$. Let $v$ be a solution of \eqref{1.1} and $\{v=0\}$ has non-empty interior.  In addition, there holds
\begin{equation*}
v_0(x)=o(|x-x_i|^{2-\frac{2}{n}})
\end{equation*}
for $x\in\partial\Omega$ as $x\to x_i$, $i=0,\cdots,N$. Then the free boundary $\partial\{v=0\}$ is strictly convex.
\end{lemma}

\begin{proof}
We show the lemma by contradiction. Suppose that 
$$[y,z]\subset \partial \{v=0\}\subset \Omega,\quad y\ne z.$$ 
Applying Corollary \ref{Corollary 2.5}, we may also assume $[x_0,y]\subseteq\partial\{v=0\}$. By translation and rotation of the coordinates, we assume $x_0=0$, 
$y=a e_n$ for some $a>0$, and $\{v=0\}\subseteq\{x_1\geq0\}$. For $\delta>0$ small, let $C_\delta$ be the convex hull of $y$ and $(B_\delta'\bigcap\{x_1<0\})\times\{\frac{a}{2}\}$. Then we have
$\det D^2 v=1,
v\geq0$
in $C_\delta$. By convexity, there holds
\begin{equation*}
\mathop{\sup}\limits_{C_\delta}v=o(\delta^{2-\frac{2}{n}})\quad as\quad\delta\to0^+,
\end{equation*}
which contradicts with that $|C_\delta|\gtrsim\delta^{n-1}$.
\end{proof}
From Lemma \ref{Lemma 3.1} and Lemma \ref{Lemma 2.6},  one knows that $v\in C^1(\Omega)$ under the assumptions of either Theorem \ref{Theorem 1.1.} or Theorem \ref{Theorem 1.2.}. In the following, we  show the $C^1$ regularity of free boundary.
\begin{lemma}
\label{Proposition 3.2.}
Under the assumptions of either Theorem \ref{Theorem 1.1.} or Theorem \ref{Theorem 1.2.}, the free boundary $\Gamma=\partial\{v=0\}$ is $C^1$.
\end{lemma}

\begin{proof}
According to Lemma \ref{Lemma 3.1} and Lemma \ref{Lemma 2.3}, the free boundary $\Gamma\backslash\{0\}$ is $C^2$. So it suffices to show the free boundary $\Gamma$ is $C^1$ near $0$.
\par Suppose not, without loss of generality, we can represent the free boundary $\Gamma$ near $0$ by
\begin{equation*}
x_n=-h(x'),\quad |x'|<l_0
\end{equation*}
Moreover, there holds
\begin{equation*}
\lim_{t\rightarrow 0^+}\frac{h(te_1)}{t}=\frac{\partial^+h}{\partial e_1}(0)=k>0.
\end{equation*}
Let $0<\delta<\frac{1}{2}$ be sufficiently small and $C>0$ be sufficiently large. For any $t>1$, let $C_t$ be the convex hull of $(te_1-kte_n)$ and $B'_\delta(e_1)$(Figure \ref{Figure 2}.). 
Define
\begin{equation*}
\bar{v}(x):=\lambda^{-2}v(\lambda x),\quad \lambda>0.
\end{equation*}
Then, for $\lambda$ small enough, $0\le \bar v\le C$ satisfies
\begin{equation*}
\det D^2\bar{v}=1 \quad \text{in}\quad C_t.
\end{equation*}
 Hence
\begin{equation*}
|C_t|\lesssim1,
\end{equation*}
which contradicts 
\begin{equation*}
\lim\limits_{t\rightarrow+\infty}|C_t|=+\infty.
\end{equation*}
Thus the free boundary $\Gamma\in  C^1$.
\end{proof}
\begin{figure}[H]
\centering
\caption{}
\label{Figure 2}
\begin{tikzpicture}[scale=2]
\fill[pattern=north east lines, pattern color=gray!70] 
    (0.7,0) -- (2,-2/3) -- (1.3,0) -- cycle;
\draw[-latex] (-1,0) -- (3.5,0) node[below] {$x'$};
\draw[-latex] (0,-1.5) -- (0,0.5) node[left] {$x_n$};
\draw[domain=-1:3.35, samples=200, thick] plot (\x, {-1/40*\x*\x-1/3*abs(\x)}) node[right] {$\lambda x_n=-h(\lambda x')$};
\draw[domain=0:3.3, thick] plot (\x, {-1/3*abs(\x)}) node[right] {$x_n=-kx_1$};
\fill (2,-2/3) circle (1pt) node[right] {$te_1-kte_n$};
\fill (1,0) circle (0.5pt) node[above] {$e_1$};
\fill (0,0) circle (1pt);
\draw (2,-2/3) -- (0.7,0);
\draw (2,-2/3) -- (1.3,0);
\draw (0.7,0) -- (1.3,0);  
\draw[dashed, thick] (3.04,-1.24) -- (2,-2/3) -- (0.8,0);
\node (C0label) at (1.7,0.25) {$C_t$};
\draw[->, thick] (C0label) -- (4/3,-2/9);
\node (Delabel) at (1.2,0.4) {$\delta$};
\draw[->] (Delabel) -- (1.15,0); 
\node at (1,-1) {$\bar{v}=0$};
\end{tikzpicture}
\end{figure}
\section{Quadratic Growth Boundary value}
We consider the case $v_0(x)\sim|x|^2$, $x\in\partial\Omega$ in this section. In the rest of the present paper, locally, we represent the free boundary $\Gamma$ as:
\begin{equation*}
x_n=-h(x'),\quad |x'|<l_0.
\end{equation*}
\subsection{Sharp Asymptotic Behavior of Free Boundary near $0$}
\begin{proposition}
\label{Proposition 4.1}
Under the assumptions of Theorem \ref{Theorem 1.1.}, we have
\begin{equation}
h(x')\sim |x'|\cdot(-\log|x'|)^{-\frac{n}{2}}.
\label{4.1}
\end{equation}
\end{proposition}
Inspired by the construction of \cite{2}(P313-314), we introduce the following sub-barrier functions. Apart from the argument in \cite{2} for dimension two, we will construct barrier functions $w_\alpha$ for all dimensions and give the lower bound of $\det D^2w_\alpha$ in exact order of $(\alpha-1)$ (see \eqref{4.6}) here.
\begin{lemma}
\label{Lemma 4.2}
For any $\alpha\in(1,\alpha_0)$, define $t_0=-1+c_1(\alpha-1)$, $s_0=\frac{n-1}{n+1}t_0+\frac{2n}{n+1}$,
\begin{equation*}
g_\alpha(t)=\begin{cases}
0,\quad t<-1\\
(t+1)^\frac{n+1}{n},\quad -1\leq t\leq t_0\\
(\frac{n+1}{2n})^2(t_0+1)^{-\frac{n-1}{n}}(t+s_0)^2,\quad t_0\leq t\leq0
\end{cases}
\end{equation*}
and
\begin{equation*}
w_\alpha(x):=|x'|^{2+\frac{2}{n}(\alpha-1)}g_\alpha(\frac{x_n}{|x'|^\alpha}),\quad x_n\leq0.
\end{equation*}
Then for some constants $\alpha_0>1$ and $\lambda,c_1>0$, depending merely on $n$, there holds that
 $w_\alpha\in C(\overline{\{x_n\leq0\}})$ is convex and
\begin{equation}
\det D^2w_\alpha\geq\lambda(\alpha-1)\chi_{\{w_\alpha>0\}}\quad \text{in}\quad  \{x_n<0\}.
\label{4.6}
\end{equation}
\end{lemma}

\begin{proof}
Let $(\alpha_0-1)$ and $c_1$ be sufficiently small. For $\alpha\in(1,\alpha_0)$, define
\begin{equation*}
\tilde w_\alpha(r,x_n):=r^{2+\frac{2}{n}(\alpha-1)}g_\alpha(\frac{x_n}{r^\alpha}),\quad r\geq0,\quad x_n\leq0.
\end{equation*}
Since $g_\alpha\in C^1((-\infty,0])$, we have $\tilde w_\alpha\in C^1((0,+\infty)\times(-\infty,0))\bigcap C([0,+\infty)\times(-\infty,0])$, and
\begin{equation*}
\partial_r\tilde{w}_\alpha=\Big(2+\frac{2}{n}(\alpha-1)\Big)r^{1+\frac{2}{n}(\alpha-1)}g_\alpha-\alpha r^{1+\frac{2}{n}(\alpha-1)}\frac{x_n}{r^\alpha}g_\alpha'\geq0.
\end{equation*}
A straightforward computation shows that
\begin{align}
\det D^2\tilde{w}_\alpha&=(\partial_{rr}^2\tilde{w}_\alpha)(\partial_{nn}^2\tilde{w}_\alpha)-(\partial_{rn}^2\tilde{w}_\alpha)^2\notag\\
&=r^{\frac{2}{n}(\alpha-1)}\biggl[\Big(2+\frac{2}{n}(\alpha-1)\Big)\Big(1+\frac{2}{n}(\alpha-1)\Big)g_\alpha\notag\\
&\qquad+\alpha\Big(\alpha-3-\frac{4}{n}(\alpha-1)\Big)\frac{x_n}{r^\alpha}g_\alpha'\notag\\
&\qquad+\alpha^2(\frac{x_n}{r^\alpha})^2g_\alpha''\biggr]\cdot r^{(\frac{2}{n}-2)(\alpha-1)}g_\alpha''\notag\\
&\qquad-r^{2(\frac{2}{n}-1)(\alpha-1)}\biggl[\Big(2-\alpha+\frac{2}{n}(\alpha-1)\Big)g_\alpha'-\alpha\frac{x_n}{r^\alpha}g_\alpha''\biggr]^2\notag\\
&=r^{2(\frac{2}{n}-1)(\alpha-1)}\biggl[\Big(2+\frac{2}{n}(\alpha-1)\Big)\Big(1+\frac{2}{n}(\alpha-1)\Big)g_\alpha''g_\alpha\notag\\
&\qquad-\alpha(\alpha-1)\frac{x_n}{r^\alpha}g_\alpha''g_\alpha'\notag\\
&\qquad-\Big(2-\alpha+\frac{2}{n}(\alpha-1)\Big)^2(g_\alpha')^2\biggr]\notag\\
&\geq0\notag
\end{align}
and
\begin{align}
\Delta\tilde{w}_\alpha&=\partial_{rr}^2\tilde{w}_\alpha+\partial_{nn}^2\tilde{w}_\alpha\notag\\
&=r^{\frac{2}{n}(\alpha-1)}\biggl[\Big(2+\frac{2}{n}(\alpha-1)\Big)\Big(1+\frac{2}{n}(\alpha-1)\Big)g_\alpha\notag\\
&\qquad+\alpha\Big(\alpha-3-\frac{4}{n}(\alpha-1)\Big)\frac{x_n}{r^\alpha}g_\alpha'\notag\\
&\qquad+\alpha^2(\frac{x_n}{r^\alpha})^2g_\alpha''\biggr]+r^{(\frac{2}{n}-2)(\alpha-1)}g_\alpha''\notag\\
&\geq0\notag
\end{align}
for $\frac{x_n}{r^\alpha}\neq-1,t_0$.\\
Hence $\tilde{w}_\alpha$ is convex, so is $w_\alpha$. And
\begin{align}
\det D^2w_\alpha(x)&=\Big((\frac{\partial_r\tilde{w}_\alpha}{r})^{n-2}\cdot \det D^2\tilde{w}_\alpha\Big)(|x'|,x_n)\notag\\
&=\biggl[\Big(2+\frac{2}{n}(\alpha-1)\Big)g_\alpha-\alpha\frac{x_n}{r^\alpha}g_\alpha'\biggr]^{n-2}\notag\\
&\qquad\cdot\biggl[\Big(2+\frac{2}{n}(\alpha-1)\Big)\Big(1+\frac{2}{n}(\alpha-1)\Big)g_\alpha''g_\alpha\notag\\
&\qquad-\alpha(\alpha-1)\frac{x_n}{r^\alpha}g_\alpha''g_\alpha'\notag\\
&\qquad-\Big(2-\alpha+\frac{2}{n}(\alpha-1)\Big)^2(g_\alpha')^2\biggr](|x'|,x_n)\notag
\end{align}
for $\frac{x_n}{|x'|^\alpha}\neq-1,t_0$.\\
Let $t=\frac{x_n}{|x'|^\alpha}$. Then for $t\in(-1,t_0)$, we have
\begin{align}
\det D^2w_\alpha(x)&=\biggl[\Big(2+\frac{2}{n}(\alpha-1)\Big)(t+1)-\frac{n+1}{n}\alpha t\biggr]^{n-2}\cdot\notag\\
&\qquad\biggl[\Big(2+\frac{2}{n}(\alpha-1)\Big)\cdot\Big(1+\frac{2}{n}(\alpha-1)\Big)\cdot\frac{n+1}{n^2}(t+1)\notag\\
&\qquad-\alpha(\alpha-1)\cdot\frac{(n+1)^2}{n^3}t-\Big(2-\alpha+\frac{2}{n}(\alpha-1)\Big)^2\cdot\frac{(n+1)^2}{n^2}(t+1)\biggr]\notag\\
&\geq \frac{1}{C(n)}(\alpha-1)-C(n)(t_0+1)\notag\\
&=\frac{1}{C(n)}\biggl[1-\Big(C(n)\Big)^2c_1\biggr](\alpha-1)\notag\\
&\geq\lambda(n)(\alpha-1)\notag
\end{align}
since $\alpha_0$ and $c_1$ are sufficiently small. And for $t\in(t_0,0)$, we have
\begin{align}
\det D^2w_\alpha(x)&=(\frac{n+1}{2n})^{2n}(t_0+1)^{1-n}(t+s_0)^{n-1}\cdot\notag\\
&\qquad\biggl[\Big(2+\frac{2}{n}(\alpha-1)\Big)s_0+(\frac{2}{n}-2)(\alpha-1)t\biggr]^{n-2}\cdot\notag\\
&\qquad\biggl[\Big((8-\frac{4}{n})(\alpha-1)-(4+\frac{8}{n^2}-\frac{16}{n})(\alpha-1)^2\Big)s_0\notag\\
&\qquad+\Big((4-\frac{4}{n})(\alpha-1)-(8+\frac{8}{n^2}-\frac{16}{n})(\alpha-1)^2\Big)t\biggr]\notag\\
&\geq(\frac{n+1}{2n})^{2n}(t_0+1)^{1-n}(t_0+s_0)^{n-1}\cdot\notag\\
&\qquad\biggl[\Big((8-\frac{4}{n})(\alpha-1)-(4+\frac{8}{n^2}-\frac{16}{n})(\alpha-1)^2\Big)s_0\notag\\
&\qquad+\Big((4-\frac{4}{n})(\alpha-1)-(8+\frac{8}{n^2}-\frac{16}{n})(\alpha-1)^2\Big)t_0\biggr]\notag\\
&\geq\lambda(n)(\alpha-1)\notag
\end{align}
since $\alpha_0$ is sufficiently small. Thus
\begin{equation*}
\det D^2w_\alpha\geq\lambda(n)(\alpha-1)\chi_{\{w_\alpha>0\}}\quad\text{in}\quad\{x_n<0\}.
\end{equation*}
\end{proof}
With this series of sub-barrier functions, we can give a lower bound of $h$.
\begin{corollary}
\label{Corollary 4.3}
Under the assumptions of Theorem \ref{Theorem 1.1.}, we have
\begin{equation}
h(x')\gtrsim |x'|\cdot(-\log|x'|)^{-\frac{n}{2}}.
\label{4.14}
\end{equation}
\end{corollary}

\begin{proof}
For any $\alpha\in(1,\alpha_0)$, let $w_\alpha$ be as in Lemma \ref{Lemma 4.2}.
 Then we have
\begin{equation*}
\det D^2w_\alpha\geq\lambda(\alpha-1)\chi_{\{w_\alpha>0\}},
\end{equation*}
\begin{equation*}
\partial_nw_\alpha\geq0,
\end{equation*}
and
\begin{equation*}
w_\alpha(x',0)\lesssim(\alpha-1)^{-\frac{n-1}{n}}|x'|^{2+\frac{2}{n}(\alpha-1)}.
\end{equation*}
Let
\begin{equation*}
\bar{w}_\alpha(x):=A^{-1}(\alpha-1)^\frac{n-1}{n}w_\alpha(x',A^n(\alpha-1)^{-\frac{n}{2}}x_n),\qquad x\in\bar{\Omega},
\end{equation*}
where $A$ is a sufficiently large constant.\\
Then we have
\begin{equation*}
\bar{w}_\alpha\leq v\quad\text{on}\quad\partial\Omega
\end{equation*}
\begin{equation*}
\det D^2\bar{w}_\alpha\geq\chi_{\{w_\alpha>0\}}\quad\text{in}\quad\Omega.
\end{equation*}
Comparison principle implies that
\begin{equation}
\bar{w}_\alpha\leq v\quad\text{in}\quad\Omega.
\label{4.20}
\end{equation}
Thus
\begin{equation}
h(x')\geq A^{-n}(\alpha-1)^\frac{n}{2}|x'|^\alpha\quad\text{for}\quad 0<|x'|<l_0.
\label{4.21}
\end{equation}
So estimate \eqref{4.14} holds for sufficiently small $|x'|$ by letting $\alpha=1-\frac{n}{2\log|x'|}$ in \eqref{4.21}.\\
Hence \eqref{4.14} holds for $0<|x'|<l_0.$
\end{proof}
Next we are going to give the upper bound of $h$.  We will use partial Legendre transformation to simplify the construction of barriers. 
\begin{lemma}
\label{Proposition 4.4.}
Under the assumptions of Theorem \ref{Theorem 1.1.}, we have
\begin{equation}
h(x')\lesssim|x'|\cdot(-\log|x'|)^{-\frac{n}{2}}.
\label{4.22}
\end{equation}
\end{lemma}

\begin{proof}
Introduce partial Legendre transformation 
\begin{equation*}
V^*(p',p_n):=\sup\limits_{x'}(p'\cdot x'-v(x',p_n))
\end{equation*}
and let
\begin{equation*}
U(p',p_n):=V^*(p',-p_n).
\end{equation*}
  Let $r_0$ be a small constant, $t_0$ be a sufficiently small constant depending on $r_0$, and $L=(B_{r_0}'(0)\backslash\{0\})\times(0,t_0)$. Then we have $U\in C(\bar{L})$, $U(\cdot,p_n)$ is convex,
\begin{equation}
\partial_{p'}U(0,p_n)=\{x'|v(x',-p_n)=0\},\quad\forall p_n
\label{4.25}
\end{equation}
and $U$ satisfies the equation
\begin{equation*}
U_{nn}+\det D_{p'}^2U=0\quad\text{in}\quad L
\end{equation*}
with  boundary conditions
\begin{equation*}
U(p',0)\gtrsim|p'|^2,\quad U\gtrsim1\quad\text{on}\quad(\partial B_{r_0}')\times[0,t_0],\quad U(0,\cdot)=0,\quad U(p',t_0)\gtrsim|p'|.
\end{equation*}
According to \eqref{4.25}, \eqref{4.22} is equivalent to
\begin{equation*}
B_{Cp_n(-\log p_n)^\frac{n}{2}}'\subseteq\partial_{p'}U(0,p_n), \quad \forall p_n\in(0,t_0),
\end{equation*}
and it suffices to construct a sub-barrier function $B(r,p_n)\in C\Big([0,r_0]\times[0,t_0]\Big)\bigcap$\\
$C^\infty\Big((0,r_0)\times(0,t_0)\Big)$ such that
\begin{equation}
B_r,B_{rr}>0,B_{nn}+(\frac{B_r}{r})^{n-2}B_{rr}\geq0\quad\text{in}\quad(0,r_0)\times(0,t_0),
\label{4.31}
\end{equation}
 and
\begin{equation}
B(0,\cdot)=0,
\label{4.32}
\end{equation}
\begin{equation}
B(|p'|,p_n)\leq U(p',p_n),\quad(p',p_n)\in\partial L,
\label{4.33}
\end{equation}
\begin{equation}
B(r,p_n)\geq Cp_n(-\log p_n)^\frac{n}{2}r,\quad(r,p_n)\in(0,r_0)\times(0,t_0).
\label{4.34}
\end{equation} 
Let
\begin{equation*}
B(r,p_n)=\theta p_n(-\log p_n)^\frac{n}{2}r+\theta^\lambda r^2f(\frac{r}{p_n(-\log p_n)^\mu\theta^\eta}).
\end{equation*}
The coefficients $\lambda,\mu,\eta,\theta$ and function $f$  will be determined  later. 
\par 
Let
\begin{equation*}
s:=\frac{r}{p_n(-\log p_n)^\mu\theta^\eta}.
\end{equation*}
A straightforward computation shows that
\begin{equation*}
\frac{B_r}{r}=\theta^{1-\eta}\cdot(-\log p_n)^{\frac{n}{2}-\mu}\cdot\frac{1}{s}+2\theta^\lambda f+\theta^\lambda sf',
\end{equation*}
\begin{equation*}
B_{rr}=2\theta^\lambda f+4\theta^\lambda sf'+\theta^\lambda s^2f'',
\end{equation*}
\begin{align}
B_{nn}&=-\frac{n}{2}\theta^{1+\eta}\cdot(-\log p_n)^{\frac{n}{2}-1+\mu}s+\frac{n}{2}(\frac{n}{2}-1)\theta^{1+\eta}\cdot(-\log p_n)^{\frac{n}{2}-2+\mu}s\notag\\
&\qquad+\theta^{\lambda+2\eta}\cdot(-\log p_n)^{2\mu}\cdot(2s^3f'+s^4f'')-\theta^{\lambda+2\eta}\mu\cdot(-\log p_n)^{2\mu-1}(3s^3f'+2s^4f'')\notag\\
&\qquad+\theta^{\lambda+2\eta}\mu\cdot(-\log p_n)^{2\mu-2}((\mu+1)s^3f'+\mu s^4f''),\notag
\end{align}
and hence
\begin{align}
B_{nn}+(\frac{B_r}{r})^{n-2}B_{rr}&=-\frac{n}{2}\theta^{1+\eta}\cdot(-\log p_n)^{\frac{n}{2}-1+\mu}s+\theta^{\lambda+2\eta}\cdot(-\log p_n)^{2\mu}\cdot(2s^3f'+s^4f'')\notag\\
&\qquad+(\theta^{1-\eta}\cdot(-\log p_n)^{\frac{n}{2}-\mu}\cdot\frac{1}{s}+2\theta^\lambda f+\theta^\lambda sf')^{n-2}\cdot\notag\\
&\qquad\theta^\lambda(2f+4sf'+s^2f'')\notag\\
&\qquad+\frac{n}{2}(\frac{n}{2}-1)\theta^{1+\eta}\cdot(-\log p_n)^{\frac{n}{2}-2+\mu}s\notag\\
&\qquad-\theta^{\lambda+2\eta}\mu\cdot(-\log p_n)^{2\mu-1}(3s^3f'+2s^4f'')\notag\\
&\qquad+\theta^{\lambda+2\eta}\mu\cdot(-\log p_n)^{2\mu-2}((\mu+1)s^3f'+\mu s^4f'').
\label{4.41}
\end{align}
To make the right-hand side of \eqref{4.41} homogeneous in $\theta$, we select $\lambda,\eta$ as
\begin{equation*}
\lambda=\frac{2}{n},\qquad\eta=1-\frac{2}{n}.
\end{equation*}
Consider the right-hand side of \eqref{4.41} as a polynomial in $(-\log p_n)$.  Then the first three terms are major terms.  
We can choose $\mu:=\frac{n}{2}-1$ such that the highest order terms in $(-\log p_n)$ are the same. 
Comparing the coefficients and choosing $t_0$ small enough, it suffices to construct a positive bounded monotone increasing smooth function $f$ such that
\begin{equation*}
2f+4sf'+s^2f''>0,
\end{equation*}
\begin{equation*}
2s^3f'+s^4f''-\frac{n}{2}s+\frac{2f+4sf'+s^2f''}{s^{n-2}}\gtrsim s+s^3f'+s^4|f''|.
\end{equation*}
It can be directly checked that
\begin{equation}
f'(t):=(t+1)^{-\frac{3}{2}}, \quad f(0)>>1
\label{4.47}
\end{equation}
satisfies the above inequalities. This implies  that \eqref{4.31} holds.  \eqref{4.32}-\eqref{4.34} just follows from the definition of $B$ and by choosing a small enough constant $\theta$. This ends the proof of the present lemma.
\end{proof}
\begin{remark}
The idea of using partial Legendre transformation comes from dimension $2$. When $n=2$, $B_{nn}+(\frac{B_r}{r})^{n-2}B_{rr}=\Delta B$, and by complex analysis,
\begin{equation*}
B(r,p_n)=\theta Re(\frac{\pi}{2}z^2+iz^2\log z)
\end{equation*}
is a suitable sub-barrier function where $z=r+ip_n$.
\end{remark}
Proposition \ref{Proposition 4.1} is a direct consequence of Corollary \ref{Corollary 4.3} and Lemma \ref{Proposition 4.4.}. We conclude this section with another corollary.
\begin{corollary}
\label{Corollary 4.5}
Under the assumptions of Theorem \ref{Theorem 1.1.}, there exists a positive constant $\mu_1>0$ such that
\begin{equation*}
v(x)\geq\mu_1|x'|^2
\end{equation*}
whenever $0<|x'|<l_0$ and $x_n>-\mu_1h(x')$.
\end{corollary}

\begin{proof}
Let $\mu_1$ be sufficiently small. Recalling \eqref{4.20}, for $0<|x'|<l_0$ and $x_n>-\mu_1h(x')$, we have
\begin{align}
v(x)&\geq\bar{w}_\alpha(x)\notag\\
&=A^{-1}(\alpha-1)^\frac{n-1}{n}w_\alpha(x',A^n(\alpha-1)^{-\frac{n}{2}}x_n)\notag\\
&=A^{-1}(\alpha-1)^\frac{n-1}{n}|x'|^{2+\frac{2}{n}(\alpha-1)}g_\alpha(\frac{A^n(\alpha-1)^{-\frac{n}{2}}x_n}{|x'|^\alpha})\notag\\
&\geq A^{-1}(\alpha-1)^\frac{n-1}{n}|x'|^{2+\frac{2}{n}(\alpha-1)}g_\alpha(-\frac{A^n(\alpha-1)^{-\frac{n}{2}}\mu_1C|x'|(-\log|x'|)^{-\frac{n}{2}}}{|x'|^\alpha})
\label{4.49}
\end{align}
for any $\alpha\in(1,\alpha_0)$.\\
Hence for such $x$ with sufficiently small $|x'|$, letting $\alpha=1-\frac{1}{\log|x'|}$ in \eqref{4.49}, we have
\begin{align}
v(x)&\geq\frac{1}{C}|x'|^2(1-C\mu_1)^2\notag\\
&\geq\mu_1|x'|^2\notag
\end{align}
since $\mu_1$ is sufficiently small.
\end{proof}
\subsection{Two Pogorelov-type Estimates}
In order to obtain $\Gamma\in C^{1,\text{Dini}}$ in $n\ge 3$, we need second order derivative estimates of $v$ near $\Gamma$. The following Pogorelov's type estimate will be useful later.
\begin{lemma}
\label{Proposition 5.1.}
Suppose $n\geq2$, $\Omega\subset\mathbb{R}^n$ is  a bounded convex domain and $\varphi\in C^2(\bar{\Omega})$ is convex.
Let $v\in C^\infty(\bar{\Omega})$ solve 
\begin{equation*}
\begin{split}
\det D^2v&=1,\quad\text{in}\quad\Omega\bigcap\{v>0\},\\
v&\geq\varphi\quad\text{on}\quad\partial\Omega.
\end{split}
\end{equation*}
Then for $\gamma>0$, there holds
\begin{equation}
\begin{split}
&(\varphi-v)_+\cdot|D^2v|\cdot(|\nabla v|-|\nabla v|_{L^\infty(\{v<0\})})_+^{n-1+\gamma}\\
\leq &C(1+\frac{1}{\gamma})\|Dv\|_{L^\infty(\Omega)}^{n+\gamma}(\|Dv\|_{L^\infty(\Omega)}+\|D\varphi\|_{L^\infty(\Omega)})
\end{split}
\label{5.1}
\end{equation}
 where $C$ is a positive constant depending merely on $n$.
\end{lemma}
\begin{proof}
Denote
\begin{equation*}
G(x,\xi):=(\varphi-v)_+\cdot v_{\xi\xi}\cdot(|\nabla v|-\beta)_+^{n-1+\gamma}(x),\quad   x\in \Omega,\quad \xi \in \mathbb S^{n-1},
\end{equation*}
where $\beta=|\nabla v|_{L^\infty(\{v<0\})}$.
Then, we may assume that
\begin{equation*}
G(x_0,\xi_0)=\max_{\overline\Omega\times \mathbb S^{n-1}} G(x,\xi)>0,\quad (x_0,\xi_0)\in (\Omega\cap \{v>0\})\times \mathbb S^{n-1}.
\end{equation*}
 Without loss of generality, we assume $\xi_0=e_1$.
After a rotation of the coordinates, we may assume that $D^2v(x_0)$ is diagonal.
Then at $(x_0,e_1)$, we have
\begin{equation}
(\log G)_i=\frac{\varphi_i-v_i}{\varphi-v}+\frac{v_{11i}}{v_{11}}+\frac{(n-1+\gamma)v_kv_{ik}}{|\nabla v|(|\nabla v|-\beta)}=0,\quad i=1,2,\cdot\cdot\cdot,n.
\label{5.6}
\end{equation}
Differentiating the equation  $\det D^2 v=1$ at $x_0$ gives that
\begin{equation}
v^{ij}v_{ijl}=0,\quad l=1,2,\cdot\cdot\cdot,n;
\label{5.8}
\end{equation}
\begin{equation}
v^{ij}v_{ij11}=v^{ii}v^{jj}v_{ij1}^2.
\label{5.9}
\end{equation}
With \eqref{5.8} and \eqref{5.9}, a direct computation yields that 
\begin{align}
0&\geq v^{ij}(\log {G})_{ij}\notag\\
&=v^{ii}\Big[\frac{\varphi_{ii}-v_{ii}}{\varphi-v}-(\frac{\varphi_i-v_i}{\varphi-v})^2+\frac{v_{11ii}}{v_{11}}-(\frac{v_{11i}}{v_{11}})^2\notag\\
&\qquad+(n-1+\gamma)\cdot\frac{(v_{ii}^2+v_kv_{iik})|\nabla v|(|\nabla v|-\beta)-(v_kv_{ik})(2v_iv_{ii}-\beta\frac{v_iv_{ii}}{|\nabla v|})}{|\nabla v|^2(|\nabla v|-\beta)^2}\Big]\notag\\
&=-\frac{n}{\varphi-v}+\frac{v^{ii}\varphi_{ii}}{\varphi-v}-v^{ii}(\frac{\varphi_i-v_i}{\varphi-v})^2+\frac{v^{ii}v^{jj}v_{ij1}^2}{v_{11}}-v^{ii}(\frac{v_{11i}}{v_{11}})^2\notag\\
&\qquad+(n-1+\gamma)\cdot\frac{v_{ii}|\nabla v|^2(|\nabla v|-\beta)-2v_i^2v_{ii}|\nabla v|+\beta v_i^2v_{ii}}{|\nabla v|^3(|\nabla v|-\beta)^2}\notag\\
&\geq-\frac{n}{\varphi-v}-v^{ii}(\frac{\varphi_i-v_i}{\varphi-v})^2+v^{11}(\frac{v_{111}}{v_{11}})^2+2\sum\limits_{i\neq1}v^{ii}(\frac{v_{11i}}{v_{11}})^2+\sum\limits_{i\neq1}\frac{(v^{ii}v_{ii1})^2}{v_{11}}\notag\\
&\qquad-v^{ii}(\frac{v_{11i}}{v_{11}})^2+(n-1+\gamma)\cdot\frac{v_{ii}|\nabla v|^2(|\nabla v|-\beta)-2v_i^2v_{ii}|\nabla v|+\beta v_i^2v_{ii}}{|\nabla v|^3(|\nabla v|-\beta)^2}.
\label{5.11}
\end{align}
Notice that
\begin{equation}
\sum\limits_{i\neq1}(v^{ii}v_{ii1})^2\geq\frac{1}{n-1}(\sum\limits_{i\neq1}v^{ii}v_{ii1})^2=\frac{1}{n-1}(\frac{v_{111}}{v_{11}})^2.
\label{5.12}
\end{equation}
Substituting \eqref{5.6} and \eqref{5.12} into \eqref{5.11}, we have
\begin{align}
0&\geq-\frac{n}{\varphi-v}-v^{ii}(\frac{\varphi_i-v_i}{\varphi-v})^2+\sum\limits_{i\neq1}v^{ii}\Big[\frac{\varphi_i-v_i}{\varphi-v}+\frac{(n-1+\gamma)v_iv_{ii}}{|\nabla v|(|\nabla v|-\beta)}\Big]^2\notag\\
&\qquad+\frac{1}{(n-1)v_{11}}\Big[\frac{\varphi_1-v_1}{\varphi-v}+\frac{(n-1+\gamma)v_1v_{11}}{|\nabla v|(|\nabla v|-\beta)}\Big]^2\notag\\
&\qquad+(n-1+\gamma)\cdot\frac{v_{ii}|\nabla v|^2(|\nabla v|-\beta)-2v_i^2v_{ii}|\nabla v|+\beta v_i^2v_{ii}}{|\nabla v|^3(|\nabla v|-\beta)^2}\notag\\
&=-\frac{n}{\varphi-v}-\frac{n-2}{(n-1)v_{11}}\cdot(\frac{\varphi_1-v_1}{\varphi-v})^2\notag\\
&\qquad+\frac{2(n-1+\gamma)v_1(\varphi_1-v_1)}{(n-1)(\varphi-v)|\nabla v|(|\nabla v|-\beta)}\notag\\
&\qquad+2(n-1+\gamma)\sum\limits_{i\neq1}\frac{v_i(\varphi_i-v_i)}{(\varphi-v)|\nabla v|(|\nabla v|-\beta)}\notag\\
&\qquad+\frac{\gamma(n-1+\gamma)v_1^2v_{11}}{(n-1)|\nabla v|^2(|\nabla v|-\beta)^2}+\frac{(n-1+\gamma)v_{11}(|\nabla v|^2-v_1^2)}{|\nabla v|^3(|\nabla v|-\beta)}\notag\\
&\qquad+(n-1+\gamma)\sum\limits_{i\neq1}\frac{v_{ii}(|\nabla v|^2-v_i^2)(|\nabla v|-\beta)+(n-2+\gamma)v_i^2v_{ii}|\nabla v|}{|\nabla v|^3(|\nabla v|-\beta)^2}\notag\\
&\geq-\frac{n}{\varphi-v}-\frac{n-2}{(n-1)v_{11}}\cdot(\frac{\varphi_1-v_1}{\varphi-v})^2\notag\\
&\qquad-2(n-1+\gamma)\frac{|\nabla\varphi-\nabla v|}{(\varphi-v)(|\nabla v|-\beta)}+\frac{\gamma v_{11}}{|\nabla v|(|\nabla v|-\beta)}.\notag
\end{align}
Hence
\begin{equation*}
\gamma(\varphi-v)v_{11}\leq C(1+\gamma)|\nabla v|(|\nabla v|+|\nabla\varphi|).
\end{equation*}
Thus
\begin{equation*}
G\leq C(1+\frac{1}{\gamma})|Dv|^{n+\gamma}(|Dv|+|D\varphi|)\quad\text{at}\quad(x_0,e_1).
\end{equation*}
 \eqref{5.1} follows immediately.
\end{proof}
The following estimate is essentially due to  \cite{2}.
\begin{lemma}
\label{Lemma 5.2}
Suppose $n\geq2$, $W$ is a bounded domain in $\mathbb{R}^n$. 
Let  $\bar{v}\in C^\infty(\overline{W})$ be convex, which solves
\begin{equation}
\det D^2\bar{v}(x)=g(x_n)(-\bar{v}_n(x))^{n+2}>0,\quad x\in W
\label{5.16}
\end{equation}
for some $g\in C^2(\mathbb{R})$. For $\forall A>10\|\nabla_{x'}\bar{v}\|_{L^\infty(W)}^2$, if $(-\bar{v})\cdot{\bar{v}_{11}}\cdot(\bar{v}_1^2+A)$ obtains its maximum at $x_0\in W$, we have
\begin{equation*}
(-\bar{v})\cdot{\bar{v}_{11}} \leq CA\quad in\quad W,
\end{equation*}
where $C$ is a universal constant.
\end{lemma}
\begin{proof}
The proof is essentially the same as that in \cite{2}. For the reader's convenience, we provide a brief proof.
Suppose 
\begin{equation*}
G(x_0)=\max_{\overline W} G(x)=\max_{\overline W}(-\bar{v})\cdot{\bar{v}_{11}}\cdot(\bar{v}_1^2+A), \quad x_0\in W.
\end{equation*}
Without loss of generality, we can assume that $\bar{v}(x_0)<0$. After a coordinate transformation
\begin{equation*}
y_1=x_1+\sum\limits_{i=2}^n\frac{\bar{v}_{1i}(x_0)}{\bar{v}_{11}(x_0)}x_i,\quad y_i=x_i(i=2,3,\cdots,n),
\end{equation*}
and then a rotation of the coordinates $y_2,\cdots,y_n$, we may assume $D^2\bar{v}$ is diagonalized and $x_0\rightarrow y_0$. 
Also, equation \eqref{5.16} is transformed to
\begin{equation}
\det D^2\bar{v}(y)=g(y\cdot\eta)\cdot(\xi\cdot\nabla\bar{v}(y))^{n+2}>0,\quad y\in U
\label{8.1}
\end{equation}
with $\eta_1=0$.
At $y_0$, we have
\begin{equation}
(\log G)_i=\frac{\bar{v}_i}{\bar{v}}+\frac{\bar{v}_{11i}}{\bar{v}_{11}}+\frac{2\bar{v}_1\bar{v}_{1i}}{\bar{v}_1^2+A}=0,\quad i=1,2,\cdots,n.
\label{8.3}
\end{equation}
Differentiating \eqref{8.1} twice gives that
\begin{equation}
\frac{\bar{v}_{ii1}}{\bar{v}_{ii}}=(n+2)\frac{\xi\cdot\nabla\bar{v}_1}{\xi\cdot\nabla\bar{v}},
\label{8.4}
\end{equation}
\begin{equation}
\frac{\bar{v}_{ii11}}{\bar{v}_{ii}}=(n+2)\frac{\xi\cdot\nabla\bar{v}_{11}}{\xi\cdot\nabla\bar{v}}-(n+2)\frac{(\xi\cdot\nabla\bar{v}_1)^2}{(\xi\cdot\nabla\bar{v})^2}+\frac{\bar{v}_{ij1}^2}{\bar{v}_{ii}\bar{v}_{jj}}.
\label{8.5}
\end{equation}
With \eqref{8.3}, \eqref{8.4}, \eqref{8.5}, one can show that
\begin{equation}
0\geq \bar v^{ii}(\log G)_{ii}=-\frac{2}{\bar{v}}-\frac{\bar{v}_1^2}{\bar{v}^2\bar{v}_{11}}+\frac{2\bar{v}_{11}(A-\bar{v}_1^2)}{(\bar{v}_1^2+A)^2}-\frac{(\xi\cdot\nabla\bar{v}_1)^2}{\bar{v}_{11}(\xi\cdot\nabla\bar{v})^2}+\sum\limits_{i\neq1}\frac{\bar{v}_{ii1}^2}{\bar{v}_{11}\bar{v}_{ii}^2}.
\label{8.6}
\end{equation}
Notice that
\begin{equation}
\sum\limits_{i\neq1}\frac{\bar{v}_{ii1}^2}{\bar{v}_{11}\bar{v}_{ii}^2}\geq\frac{1}{(n-1)\bar{v}_{11}}(\sum\limits_{i\neq1}\frac{\bar{v}_{ii1}}{\bar{v}_{ii}})^2=\frac{1}{(n-1)\bar{v}_{11}}\Big((n+2)\frac{\xi\cdot\nabla\bar{v}_1}{\xi\cdot\nabla\bar{v}}-\frac{\bar{v}_{111}}{\bar{v}_{11}}\Big)^2.
\label{8.7}
\end{equation}
Substituting \eqref{8.7} and \eqref{8.3} into \eqref{8.6}, we have
\begin{align}
0\geq&-\frac{2}{\bar{v}}-\frac{\bar{v}_1^2}{\bar{v}^2\bar{v}_{11}}+\frac{2\bar{v}_{11}(A-\bar{v}_1^2)}{(\bar{v}_1^2+A)^2}\notag\\
&+\frac{3(n+2)(\xi\cdot\nabla\bar{v}_1)^2}{(n-1)\bar{v}_{11}(\xi\cdot\nabla\bar{v})^2}\notag\\
&+\frac{2(n+2)}{(n-1)\bar{v}_{11}}\cdot\frac{\xi\cdot\nabla\bar{v}_1}{\xi\cdot\nabla\bar{v}}\cdot(\frac{\bar{v}_1}{\bar{v}}+\frac{2\bar{v}_1\bar{v}_{11}}{\bar{v}_1^2+A})\notag\\
&+\frac{1}{(n-1)\bar{v}_{11}}\cdot(\frac{\bar{v}_1}{\bar{v}}+\frac{2\bar{v}_1\bar{v}_{11}}{\bar{v}_1^2+A})^2\notag\\
\geq&-\frac{2}{\bar{v}}-\frac{\bar{v}_1^2}{\bar{v}^2\bar{v}_{11}}+\frac{2\bar{v}_{11}(A-\bar{v}_1^2)}{(\bar{v}_1^2+A)^2}\notag\\
&-\frac{1}{3\bar{v}_{11}}\cdot(\frac{\bar{v}_1}{\bar{v}}+\frac{2\bar{v}_1\bar{v}_{11}}{\bar{v}_1^2+A})^2\notag\\
\geq&-\frac{2}{\bar{v}}-\frac{5\bar{v}_1^2}{3\bar{v}^2\bar{v}_{11}}+\frac{2\bar{v}_{11}(A-\frac{7}{3}\bar{v}_1^2)}{(\bar{v}_1^2+A)^2}.
\label{8.8}
\end{align}
Since $A>10\bar{v}_1^2$, \eqref{8.8} implies that
\begin{equation*}
(-\bar{v})\cdot\bar{v}_{11}\leq CA, \quad \text{at} \quad y_0.
\end{equation*}
This ends the proof of the present lemma.
\end{proof}
\begin{remark}
\label{Remark 5.3}
In \cite{2}, Savin used the  auxiliary function $G(x)=(-\bar{v})\cdot\bar{v}_{11}\cdot e^{\frac{\sigma}{2}\bar{v}_1^2}$. 
However, for later use, we need $\bar v_{11}$ controlled by a polynomial of $v_1$. Hence, we replace $e^{\frac{\sigma}{2}\bar{v}_1^2}$ by $\bar v_1^2+A$.
\end{remark}
\subsection{The Optimal Regularity of Free Boundary}
Before the proof of Theorem \ref{Theorem 1.1.}, we need the following estimate of $D^2h$ first.
\begin{proposition}
\label{Proposition 6.1.}
Under the assumptions of Theorem \ref{Theorem 1.1.}, there exists a positive constant $C$ such that
\begin{equation}
|D^2h(x_0')|\leq C|x_0'|^{-C}\quad\text{for any}\quad 0<|x_0'|<l_0.
\label{6.1}
\end{equation}
\end{proposition}

\begin{proof} It suffices to show \eqref{6.1} for sufficiently small $|x_0'|$. Let $X_0:=\Big(x_0',-h(x_0')\Big)$. 
We divide the proof into three steps.
\par \textit{Step1. Construction of smooth approximate solutions $v^\varepsilon$.}\\
Let $v^\varepsilon$ solve
\begin{equation*}
\begin{cases}
\det D^2v^\varepsilon=a^\varepsilon(v^\varepsilon)\quad &\text{in}\quad\Omega\\
v^\varepsilon=v_0\quad &\text{on}\quad\partial\Omega
\end{cases}.
\end{equation*}
Here $\varepsilon\le a^\varepsilon(t)\le 1$ is a smooth function such that $a^\varepsilon(t)=1$, $t\in (0,+\infty)$; $a^\varepsilon(t)=\varepsilon$, $t\in (-\infty,-\varepsilon)$. 
Then $v^\varepsilon\in C^\infty(\Omega)$ satisfies $v^\varepsilon\leq v$  in $\Omega$.
From Lemma \ref{Lemma 2.1}, up to a subsequence, $v^{\varepsilon}\rightarrow v$ uniformly.
\par \textit{Step 2. Pogorelov-type estimate.}\\
Let 
\begin{equation}
H^+:=\{x_n>-2h(x_0')-\frac{4h(x_0')}{|x_0'|^2}x_0'\cdot(x'-x_0')\},\quad H:=\partial H^+.
\label{6.10}
\end{equation}
Let $\mu_1$ be given by Corollary \ref{Corollary 4.5}. Define
\begin{equation*}
\tilde{\Omega}:=\Omega\bigcap\{v<\frac{\mu_1}{8}|x_0'|^2\}\bigcap H^+.
\end{equation*}
Using \eqref{4.1}, \eqref{6.10} and Corollary \ref{Corollary 4.5}, by a direct computation, there exists a universal constant $\delta_0<\frac{1}{2}$ such that
\begin{equation}
\overline{\tilde{\Omega}}\subseteq\Omega\bigcap B_{|x_0'|^{\delta_0}}\bigcap\Big\{y\Big||y'|\geq\frac{1}{2}|x_0'|,\quad y_n\leq-\mu_1h(y')\Big\},
\label{6.12}
\end{equation}
\begin{equation}
B_{|x_0'|^{2\theta}}(\overline{\tilde{\Omega}}+[-2|x_0'|^\theta,-|x_0'|^\theta]\nu)\subseteq\{v=0\},
\label{6.13}
\end{equation}
where $\theta=\frac{\delta_0}{4n^2}$ and $\nu=(\frac{4h(x_0')}{|x_0'|^2}x_0',1)$ is the normal of the hyperplane $H$. Here $\overline{\tilde{\Omega}}\subseteq\{|y'|\geq\frac{1}{2}|x_0'|\}$ is a direct consequence of $\tilde{\Omega}\subseteq H^+$, $\overline{\tilde{\Omega}}\subseteq\{y_n\leq-\mu_1h(y')\}$ then follows from Corollary \ref{Corollary 4.5}. And a direct computation with \eqref{4.1}, \eqref{6.10} implies that $\overline{\tilde{\Omega}}\subseteq B_{|x_0'|^{\delta_0}}$, then \eqref{6.13} follows by \eqref{4.1}.

Let $\tilde{H}^+:=H^++\frac{h(x_0')}{2}e_n$, $\tilde{H}=\partial\tilde{H}^+$. 
From \eqref{6.12}-\eqref{6.13}, for a suitable large constant $C$, there exists an affine function $\varphi$ such that 
\begin{equation*}
\varphi(x)=|x_0'|^C,\quad x\in \tilde{H};\quad \varphi(x)=-|x_0'|^C,\quad x\in H;\quad \varphi(x)\leq\frac{\mu_1}{16}|x_0'|^2,\quad x\in \tilde \Omega.
\end{equation*}
From \eqref{6.12}, we know \begin{equation*}
\partial \tilde \Omega\subset\{v=\frac{\mu_1}{8}|x_0'|^2\}\cup H,\quad \partial\tilde \Omega\cap \partial \Omega=\emptyset.
\end{equation*}
Since $v^\varepsilon\rightarrow v$ uniformly in $\overline \Omega$, from the choice of $\varphi$, we know 
\begin{equation*}
v^\varepsilon\ge \varphi,\quad \text{on}\quad \partial \tilde \Omega
\end{equation*}
for $\varepsilon$ small enough. This allows us to apply Lemma \ref{Proposition 5.1.} to $v^\varepsilon$ in $\tilde \Omega$ to get
\begin{equation}
(\varphi-v^{\varepsilon})_+\cdot|D^2v^{\varepsilon}|\cdot(|\nabla v^{\varepsilon}|-|\nabla v^\varepsilon|_{L^\infty(\{v^\varepsilon<0\}\cap \tilde \Omega)})_+^{2n-2}\leq C_0|x_0|^{-C_0}\quad\text{in}\quad\tilde{\Omega}
\label{6.18}
\end{equation}
by taking $\gamma=n-1$ in Lemma \ref{Proposition 5.1.}. 
\par  \textit{Step 3. Hodograph transformation.}\\
For another large constant $A>>C$, define
\begin{equation*}
U:=\Omega\bigcap\{v<|x_0'|^A\}\bigcap\tilde{H}^+.
\end{equation*}
Then $\bar{U}\subseteq\tilde{\Omega}\subset\subset\Omega$.  By the definition of $U$ and \eqref{6.13}, we know 
\begin{equation}\label{1004-1}
v(x)=|x_0'|^A,\quad \text{on}\quad \partial U\backslash\tilde{H},\quad v(x)=0,\quad \text{on}\quad (\partial U\backslash\tilde{H})-|x_0'|^\theta \nu. 
\end{equation}
Hence, for sufficiently small $\varepsilon$, by convexity and \eqref{1004-1}, we have the following
\begin{equation}\label{1005-1}
Dv^\varepsilon\cdot \nu\gtrsim |x_0'|^{A-\theta},\quad \text{on}\quad \partial U\backslash\tilde{H}. 
\end{equation}
From the definition of $\varphi$, we also know $\varphi_{x_n}>0$ which implies $\varphi(x)\ge |x_0'|^C$ in $\tilde H^+$. 
Therefore, from $|\nabla v^\varepsilon|_{L^\infty(\{v^\varepsilon<0\}\cap \tilde \Omega)}\rightarrow 0$ as $\varepsilon \rightarrow 0$, \eqref{6.18} gives that 
\begin{equation}\label{1005-2}
|D^2v^{\varepsilon}|\leq C_0|x_0'|^{-C_0}\quad\text{on}\quad\partial U\backslash\tilde{H}
\end{equation}
for $\varepsilon$ small enough. 

In the following, we use a new orthonormal coordinate system $y=(y',y_n)$ where $\tilde H^+=\{y_n>0\}$. From Lemma \ref{Lemma 2.1}, we have 
\begin{equation}
[v^\varepsilon]_{C^\frac{1}{n}(\overline{\{v^\varepsilon<-\varepsilon\}})}\leq C\varepsilon^\frac{1}{n}.
\label{6.9}
\end{equation}
From \eqref{6.12} and $\max_{\overline U}v^\varepsilon>0$,  we can use \eqref{6.9} to get
\begin{equation}\label{1004-5}
v^\varepsilon\geq-\varepsilon-C\varepsilon^\frac{1}{n}|x_0'|^\frac{\delta_0}{n},\quad \text{in}\quad U. 
\end{equation}
From \eqref{6.9}, we also know 
\begin{equation}\label{1004-2}
\inf_{\Omega}v^\varepsilon \sim -\varepsilon^{\frac 1n}. 
\end{equation}
By \eqref{6.13} and \eqref{1004-2}, using the convexity of $v^\varepsilon$, 
we have 
\begin{equation}\label{1004-3}
v^\varepsilon\lesssim-\varepsilon^\frac{1}{n}|x_0'|^{2\theta}, \quad \text{in}\quad \bar{U}-\frac{3}{2}|x_0'|\nu-\frac 12|x_0'|^{2\theta} B_1(0).
\end{equation}
Combining \eqref{1004-5} and \eqref{1004-3}, we have $v^\varepsilon_{y_n}>0$ in $\bar{U}$. Then, we can express the graph of $v^{\varepsilon}|_{\bar{U}}$  by
\begin{equation*}
(y',-\tilde{v}^{\varepsilon}(y',y_{n+1}),y_{n+1}),\quad(y',y_{n+1})\in\bar{W}
\end{equation*}
with some bounded domain $W$ and locally convex smooth function $\tilde{v}^{\varepsilon}$ on $\bar{W}$. Following the calculations in \cite[Lemma 2.2]{1} or \cite[Proposition 4.3]{2}, one obtains that 
\begin{equation}
\nabla_{y'}\tilde{v}^{\varepsilon}(y',y_{n+1})=\frac{1}{v^{\varepsilon}_n}\nabla_{y'}v^{\varepsilon}(y',-\tilde{v}^{\varepsilon}(y',y_{n+1})),
\label{6.22}
\end{equation}
\begin{equation*}
\det D^2\tilde{v}^{\varepsilon}(y',y_{n+1})=(-\tilde{v}^{\varepsilon}_{n+1}(y',y_{n+1}))^{n+2}\cdot a^{\varepsilon}(y_{n+1})\quad\text{for}\quad(y',y_{n+1})\in W.
\end{equation*}
Then, from \eqref{1004-5}, \eqref{1004-3} and convexity of $v^\varepsilon$, for any $\tilde y'\in \mathbb S^{n-2}$, we have 
\begin{equation*}
\nabla_y  v^{\varepsilon}(y',y_n)\cdot (\frac 12|x_0'|^{2\theta}\tilde y',\frac{3}{2}|x_0'|)\ge v^{\varepsilon}(y',y_n)-v^{\varepsilon}(y'-\frac 12|x_0'|^{2\theta}\tilde y',y_n-\frac{3}{2}|x_0'|)\ge 0.
\end{equation*}
This implies 
\begin{equation*}
|\nabla_{y'}\tilde{v}^{\varepsilon}|\leq C|x_0'|^{-C}\quad\text{in}\quad W.
\end{equation*}
Now we consider the quantity \begin{equation*}
G(y)=(-\tilde{v}^{\varepsilon})\cdot{\tilde{v}^{\varepsilon}_{11}}\cdot((\tilde{v}^{\varepsilon}_1)^2+10C|x_0'|^{-C}),\quad y\in W.
\end{equation*}
Let $\bar y$ be the maximum point of $G(y)$ in $\overline W$. 
\begin{itemize}
\item[(1).] $\bar y\in \partial W$ and $\tilde{v}^{\varepsilon}(\bar y)=0$.  Then $G(\bar y)=0$.
\item[(2).] $\bar y\in \partial W$ and $v(\bar{y}',-\tilde{v}^\varepsilon(\bar{y}))=|x_0'|^A$.  A direct computation gives that 
\begin{equation*}
\widetilde v_{11}^\varepsilon(y',y_{n+1})=\frac{v_{11}^\varepsilon (v_n^\varepsilon)^2-2v_{n1}^\varepsilon v_1^\varepsilon v_n^\varepsilon+v_{nn}^\varepsilon (v_1^\varepsilon)^2
}{(v_n^\varepsilon)^3}.
\end{equation*}
Then we can use \eqref{1005-1} and \eqref{1005-2} to get  $G(\bar y)\le C_0 |x_0'|^{-C_0}$. 
\item[(3).] $\bar y\in W$. Then we can apply Lemma \ref{Lemma 5.2} to get $G(\bar y)\le C_0 |x_0'|^{-C_0}$.
\end{itemize}
From (1)-(3) above, we know 
\begin{equation*}
(-\tilde v^{\varepsilon})\tilde{v}_{11}^{\varepsilon}\le C_0 |x_0'|^{-C_0},\quad \text{in}\quad W. 
\end{equation*}
In particular, near $(X_0',v^\varepsilon(X_0))$, we get
\begin{equation*}
\tilde{v}_{11}^{\varepsilon}\leq C_0|x_0'|^{-C_0}
\end{equation*}
 This implies that the principle curvature $\kappa_1$ of $\Gamma$ at $X_0$ satisfies
\begin{equation*}
\kappa_1\leq C_0|x_0'|^{-C_0}.
\end{equation*}
Similar estimates also hold for $\kappa_i$, $i=2,\cdots,n-1$. 
Thus \eqref{6.1} holds.
\end{proof}
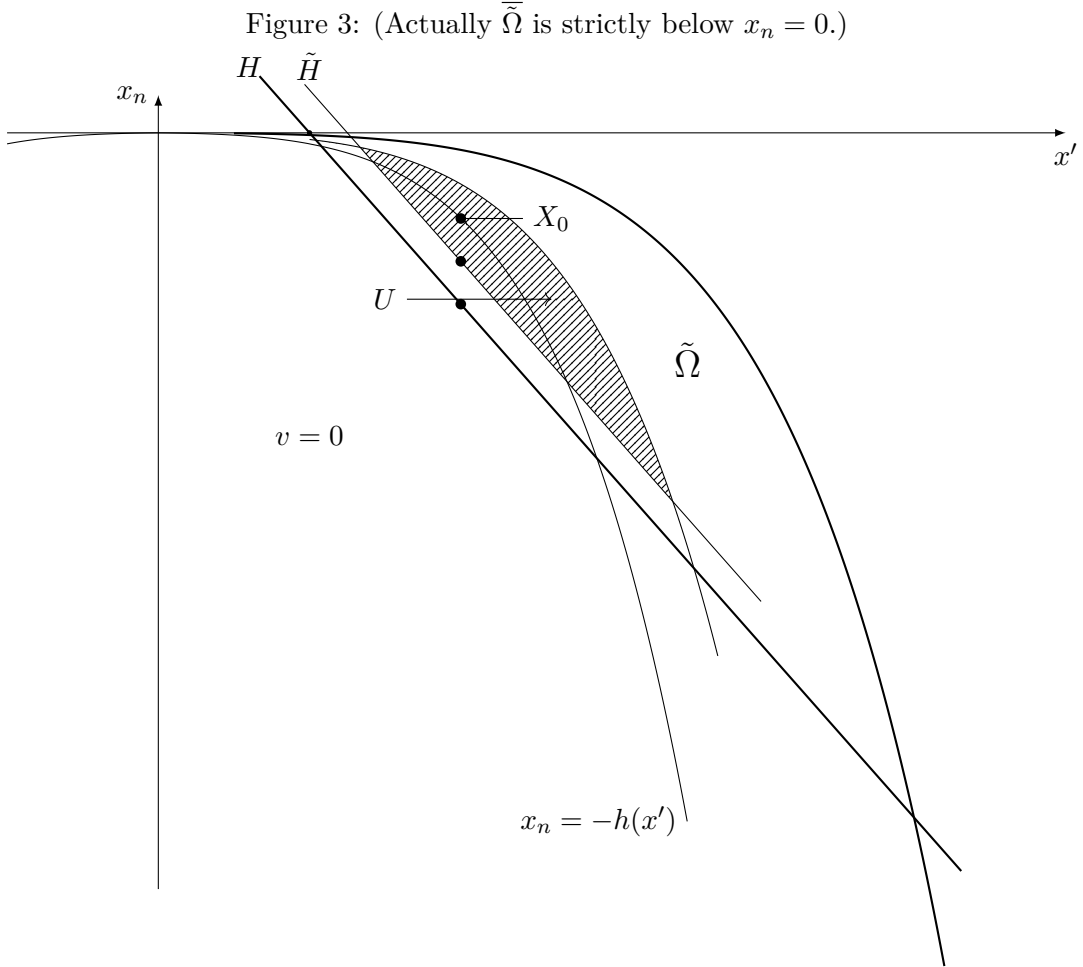
\begin{figure}[H]
\centering
\caption{(Actually $\overline{\tilde{\Omega}}$ is strictly below $x_n=0$.)}
\label{3}
\begin{tikzpicture}[scale=2]
\draw[-latex] (-1,0) -- (6,0) node[below] {$x'$};
\draw[-latex] (0,-5) -- (0,0.25) node[left] {$x_n$};
\path[name path=curve] plot[domain=1:4, samples=200]
(\x, {-0.3/384*(2+abs(\x))^4*abs(\x)*(-ln(abs(\x)/8))^(-1/2)});
\path[name path=line] (1,{1/3*(ln(4))^(-1/2)}) -- (4,{-11/3*(ln(4))^(-1/2)});
\path[name intersections={of=curve and line, name=int}];
\fill[pattern=north east lines, pattern color=black]
(int-1) -- plot[domain=1.043:3.534, samples=200]
(\x, {-0.3/384*(2+abs(\x))^4*abs(\x)*(-ln(abs(\x)/8))^(-1/2)})
-- (int-2) -- cycle;
\node(U0label) at (1.5,-1.1) {$U$};
\draw [->] (U0label) -- (2.6,-1.1);
\draw[domain=-1:3.5, samples=200] plot
(\x, {-0.5/384*(2+abs(\x))^4*abs(\x)*(-ln(abs(\x)/8))^(-1/2)})
node[left] {$x_n=-h(x')$};
\draw[domain=0.5:5.2, samples=200, thick] plot
(\x, {-0.1/384*(2+abs(\x))^4*abs(\x)*(-ln(abs(\x)/8))^(-1/2)});
\draw[domain=1:3.7, samples=200] plot
(\x, {-0.3/384*(2+abs(\x))^4*abs(\x)*(-ln(abs(\x)/8))^(-1/2)});
\draw[shorten >= -10cm, shorten <= -1cm, thick] (1,0) --
(2,{-4/3*(ln(4))^(-1/2)}) node[above] at (0.6,0.3) {$H$};
\draw[shorten >= -6cm, shorten <= -0.1cm]
(1,{1/3*(ln(4))^(-1/2)}) --
(2,{-(ln(4))^(-1/2)}) node[above] at (1,{1/3*(ln(4))^(-1/2)}) {$\tilde{H}$};
\node at (3.5,-1.5) {\Large$\tilde{\Omega}$};
\fill (1,0) circle (0.5pt);
\fill (2,{-4/3*(ln(4))^(-1/2)}) circle (1pt);
\fill (2,{-2/3*(ln(4))^(-1/2)}) circle (1pt);
\fill (2,{-(ln(4))^(-1/2)}) circle (1pt);
\node(X0label) at (2.6,{-2/3*(ln(4))^(-1/2)}) {$X_0$};
\draw [->] (X0label) -- (2,{-2/3*(ln(4))^(-1/2)});
\node at (1,-2) {$v=0$};
\end{tikzpicture}
\end{figure}
Now we are prepared to prove Theorem \ref{Theorem 1.1.}.
\begin{proof}[Proof of Theorem 1.1]
Recalling Lemma \ref{Proposition 3.2.}, the free boundary $\Gamma$ is $C^1$.
For $r>0$, denote
\begin{equation*}
 \omega(r)=\sup\limits_{|y-x|<r}|\nabla_{x'}h(y)-\nabla_{x'}h(x)|.
\end{equation*}
By convexity and \eqref{4.1}, we have $\omega(r)\gtrsim (-\log r)^{-\frac{n}{2}}$.
In the following, we estimate the upper bound of $\omega(r)$. Let $C>1$ be as in Proposition \ref{Proposition 6.1.}. Then for $|x_0'|,|y_0'|<l_0$ with $|x_0'-y_0'|<r$, there are two cases.
\par  Case I: $|x_0'|\geq r^\frac{1}{2C}$.\\
Proposition \ref{Proposition 6.1.} implies $|D^2h|\lesssim r^{-\frac{1}{2}}\lesssim r^{-1}(-\log r)^{-\frac{n}{2}}$ on $[x_0,y_0]$, and
\begin{equation*}
|\nabla_{x'}h(y_0')-\nabla_{x'}h(x_0')|\lesssim(-\log r)^{-\frac{n}{2}}.
\end{equation*}
\par Case II: $|x_0'|<r^{\frac{1}{2C}}$.\\
From convexity, \eqref{4.1} implies $|\nabla_{x'}h(x_0')|,|\nabla_{x'}h(y_0')|\lesssim(-\log r)^{-\frac{n}{2}}$, and
\begin{equation*}
|\nabla_{x'}h(y_0')-\nabla_{x'}h(x_0')|\lesssim(-\log r)^{-\frac{n}{2}}.
\end{equation*}
\par Combining the lower bound of $\omega(r)$, Case I and Case II, one knows
\begin{equation*}
\omega(r)\sim(-\log r)^{-\frac{n}{2}}. 
\end{equation*}
Hence $\int\nolimits_0^{l_0}\frac{\omega(r)}{r}dr=+\infty$ when $n=2$, while $\int\nolimits_0^{l_0}\frac{\omega(r)}{r}dr<+\infty$ when $n\geq3$. This ends the proof of Theorem \ref{Theorem 1.1.}.
\end{proof}
\section{Super-quadratic Growth Boundary value}
This section is devoted to the proof Theorem \ref{Theorem 1.2.}. 
Sharp asymptotic estimates on $\Gamma$ near $0$ are needed. 
\begin{proposition}
\label{Proposition 7.1}
Under the assumption of Theorem \ref{Theorem 1.2.}, we have
\begin{equation*}
h(x')\sim |x'|^\alpha\quad\text{for}\quad 0<|x'|<l_0.
\end{equation*}
\end{proposition}
\par 
Similarly to Lemma \ref{Lemma 4.2} in Section 4, we need to construct suitable barrier functions first. Since the barrier functions constructed with the method in \cite{2} in Lemma \ref{Lemma 4.2} were valid only for sufficiently small $(\alpha-1)$ in high dimensions, we need some modifications of the function $g$ here.
\begin{lemma}\label{Lemma 7.2} 
For $\forall \alpha>1$, there exists a convex function $w$ such that 
\begin{equation*}
\begin{split}
 \lambda\chi_{\{w>0\}} \le &\det D^2w\le \Lambda\chi_{\{w>0\}}\quad\text{in}\quad\{x_n<0\}\\
\lambda |x'|^{2+\frac{2}{n}(\alpha-1)}&\le w(x',0)\le  \Lambda |x'|^{2+\frac{2}{n}(\alpha-1)}
 \end{split}
\end{equation*}
where $\lambda, \Lambda$ are positive constants depending only on $n$ and $\alpha$. Moreover,  $w_n\ge 0$  holds and 
\begin{equation*}
\{w=0\}=\{x_n\le  -|x'|^\alpha\}.
\end{equation*}
\end{lemma}

\begin{proof}
Let $t_0$ be a constant larger than $-1$, but close enough to $-1$. Consider a function $g\in  C\Big((-\infty,0]\Big)\bigcap C^1\Big((-\infty,0)\Big)$ such that $g(t)=(t+1)_+^\frac{n+1}{n}$ for $t\in(-\infty,t_0]$. 
In the following, we determine $g$ in $[t_0,0)$.  Define
\begin{equation*}
\tilde{w}(r,x_n):=
\begin{cases}
r^{2+\frac{2}{n}(\alpha-1)}g(\tilde s),\qquad&r>0,x_n\leq0\\
0,\qquad&r=0,x_n\leq0
\end{cases}, \quad \tilde s=\frac{x_n}{r^\alpha}. 
\end{equation*}
A straightforward computation shows that
\begin{equation*}
\frac{\tilde{w}_r}{r}=r^{\frac{2}{n}(\alpha-1)}\biggl[\Big(2+\frac{2}{n}(\alpha-1)\Big)g(\tilde{s})-\alpha\tilde{s}g'(\tilde{s})\biggr],
\end{equation*}
\begin{align}
\det D^2\tilde{w}(r,x_n)&=r^{-\frac{2n-4}{n}(\alpha-1)}\Biggl[\Big(2+\frac{2}{n}(\alpha-1)\Big)\Big(1+\frac{2}{n}(\alpha-1)\Big)g''g\notag\\
&\qquad-\alpha(\alpha-1)\tilde{s}g''g'-\Big(1-\frac{n-2}{n}(\alpha-1)\Big)^2(g')^2\Biggr],\notag
\end{align}
for $r>0$, $x_n<0$ and $\tilde{s}\neq t_0,-1$. Define $w(x):=\tilde{w}(|x'|,x_n)$, then $w\in C^1(\{x_n<0\})$, and
\begin{align}
\det D^2w(x)=&\Big((\frac{\tilde{w}_r}{r})^{n-2}\cdot\det D^2\tilde{w}\Big)(|x'|,x_n)\notag\\
=&\Biggl[\Big(2+\frac{2}{n}(\alpha-1)\Big)g(s)-\alpha sg'(s)\Biggr]^{n-2}\cdot\notag\\
&\Biggl[\Big(2+\frac{2}{n}(\alpha-1)\Big)\Big(1+\frac{2}{n}(\alpha-1)\Big)g''(s)g(s)-\alpha(\alpha-1)sg''(s)g'(s)\notag\\
&-\Big(1-\frac{n-2}{n}(\alpha-1)\Big)^2\cdot\Big(g'(s)\Big)^2\Biggr]
\label{9.1}
\end{align}
for $s\neq-1,t_0$, where $s=\frac{x_n}{|x'|^\alpha}$. In particular,
\begin{align}
\det D^2w(x)&=\frac{n+1}{n^2}\Biggl[\Big(2+\frac{2}{n}(\alpha-1)\Big)(s+1)-\frac{n+1}{n}\alpha s\Biggr]^\frac{2-n}{n}\cdot\notag\\
&\qquad\Biggl[\Big(2+\frac{2}{n}(\alpha-1)\Big)\Big(1+\frac{2}{n}(\alpha-1)\Big)(s+1)\notag\\
&\qquad-\alpha(\alpha-1)s\cdot\frac{n+1}{n}-\Big(1-\frac{n-2}{n}(\alpha-1)\Big)^2(n+1)(s+1)\Biggr]\sim1
\label{7.8}
\end{align}
for $s\in(-1,t_0)$, since $t_0$ is close enough to $-1$.
\par The remaining part of the proof can be transformed to an initial value problem for first order ODE systems. Let $\lambda^*$ be a sufficiently small positive constant and consider the initial value problem for ODE systems
 \begin{equation}
 \left\{
 \begin{aligned}
 &g_1'(t)=g_2(t)\\
 &g_2'(t)=\frac{\lambda^*+\Big(1-\frac{n-2}{n}(\alpha-1)\Big)^2\cdot(g_2(t))^2}{\Big(2+\frac{2}{n}(\alpha-1)\Big)\Big(1+\frac{2}{n}(\alpha-1)\Big)g_1(t)-\alpha(\alpha-1)tg_2(t)}\\
 &g_1(t_0)=(1+t_0)^\frac{n+1}{n}\\
 &g_2(t_0)=\frac{n+1}{n}(t_0+1)^\frac{1}{n}.
 \end{aligned}\\
 \right.
 \label{7.14}
 \end{equation}
The solutions of \eqref{7.14} exist in $[t_0,0)$. In fact, once we prove
\begin{equation}
g_2(t)<C(-t)^{-\varepsilon},\quad\forall t\in(t_0,0)
\label{7.15}
\end{equation}
for some constant $C>0$ and $\varepsilon\in(0,1)$, let $g=g_1$ in $[t_0,0)$, then $g\in C\Big((-\infty,0]\Big)\bigcap$\\$C^1\Big((-\infty,0)\Big)$, $\Big(2+\frac{2}{n}(\alpha-1)\Big)g(s)-\alpha sg'(s)\sim1$ for $s\in[t_0,0)$, and by the second equality in \eqref{7.14}, $\Big(2+\frac{2}{n}(\alpha-1)\Big)\Big(1+\frac{2}{n}(\alpha-1)\Big)g''(s)g(s)-\alpha(\alpha-1)sg''(s)g'(s)$\\$-\Big(1-\frac{n-2}{n}(\alpha-1)\Big)^2\cdot\Big(g'(s)\Big)^2=\lambda^*$ for $s\in(t_0,1)$. By \eqref{9.1}, we have
\begin{equation}
\det D^2w\sim1\quad\text{for}\quad s\in(t_0,0).
\label{9.2}
\end{equation}
And by \eqref{7.8} \eqref{9.2}, and the construction of $w$, $w$ satisfies all the requirements in Lemma \ref{Lemma 7.2}. So it suffices to prove \eqref{7.15}.\\
Case I: $\alpha\geq2$.\\
Let
\begin{equation*}
\varepsilon\in\biggl(\frac{\Big(1-\frac{n-2}{n}(\alpha-1)\Big)^2}{\alpha(\alpha-1)},1\biggr)
\end{equation*}
Then
\begin{equation*}
g_2'(t)<\frac{\varepsilon g_2(t)}{-t},\quad\forall t\in[t_0,0),
\end{equation*}
thus
\begin{equation*}
g_2(t)\leq (-t)^{-\varepsilon}(-t_0)^\varepsilon g_2(t_0),\quad\forall t\in[t_0,0),
\end{equation*}
and the proof is completed.\\
Case II: $1<\alpha<2$.\\
Then
\begin{equation*}
g_1''g_1<(g_1')^2\quad\text{in}\quad(t_0,0).
\end{equation*}
Thus $\frac{g_1'}{g_1}$ is non-increasing in $[t_0,0)$, and $g_1'\leq C^*g_1$ in $[t_0,0)$ for some positive constant $C^*$. Then $e^{-C^*t}g_1(t)$ in non-increasing in $[t_0,0)$, and $g_2=g_1'$ is bounded in $[t_0,0)$, the proof is also completed.
\end{proof}
With the barrier function, we can now show the sharp asymptotic behavior of free boundary near $0$ now.
\begin{proof}[Proof of Proposition \ref{Proposition 7.1}]
To show the lower bound $h(x')\gtrsim|x'|^\alpha$, the approach applied in the proof of Corollary \ref{Corollary 4.3} would be still valid and sufficient.\\
To show the upper bound $h(x')\lesssim|x'|^\alpha$, let
\begin{equation*}
\tilde{v}(x):=
\begin{cases}
v(x)&,\quad x_n>-h(x')\\
0&,\quad x_n\leq-h(x')
\end{cases}
\end{equation*}
be a function on $\overline{B_{\frac{l_0}{2}}'}\times(-\infty,0]$. Then $\tilde{v}$ is an Alexandrov solution of
\begin{equation*}
\det D^2\tilde{v}=\chi_{\{\tilde{v}>0\}}\quad\text{in}\quad B_{\frac{l_0}{2}}'\times(-\infty,0).
\end{equation*}
Let $L=2\sup\limits_{\partial B_{\frac{l_0}{2}}'}h$, $w$ be as in Lemma \ref{Lemma 7.2}, and
\begin{equation*}
\tilde{w}(x):=Aw(x',A^{-n}x_n),
\end{equation*}
where $A$ is a sufficiently large constant. Then, in $B_{\frac{l_0}{2}}'\times(-L,0)$, we have
\begin{equation*}
\det D^2\tilde{v}=\chi_{\{\tilde{v}>0\}},\qquad\det D^2\tilde{w}\leq\chi_{\{\tilde{w}>0\}}
\end{equation*}
and we have
\begin{equation*}
\tilde{v}\leq\tilde{w}\quad\text{on}\quad\partial\Big(B_{\frac{l_0}{2}}'\times(-L,0)\Big).
\end{equation*}
Hence
\begin{equation*}
\tilde{v}\leq\tilde{w}\quad\text{in}\quad B_{\frac{l_0}{2}}'\times(-L,0)
\end{equation*}
by comparison principle and thus
\begin{equation*}
h(x')\lesssim|x'|^\alpha.
\end{equation*}
\end{proof}
Next we prove the main result, Theorem \ref{Theorem 1.2.}. We just sketch the proof since it is based on standard blow-up argument.
\begin{proof}[Proof of Theorem 1.2]
It suffices to show
\begin{equation}
|D^2h(x')|\lesssim |x'|^{\alpha-2}\quad\text{for}\quad0<|x'|<l_0.
\label{7.27}
\end{equation}
Let
\begin{equation*}
v^r(x):=r^{-2-\frac{2}{n}(\alpha-1)}v(rx',r^\alpha x_n),\quad r>0.
\end{equation*}
In fact, \eqref{7.27} can be obtained based on the following observations.
\begin{itemize}
\item For any $\{r_k\}_{k=1}^\infty\subseteq(0,\frac{1}{2})$ with $r_k\rightarrow0$, passing to the limit if necessary, $v^{r_k}$ converges locally uniformly to a half-space solution $\tilde{v}$ to the obstacle problem $\det D^2v=\chi_{\{\tilde{v}>0\}}$ in $\{x_n<0\}$. And $\partial\{\tilde{v}=0\}$ is the graph of $-\tilde{h}$ for some convex function $h$ defined on $\mathbb{R}^{n-1}$ with $\tilde{h}(x')\sim|x'|^\alpha$.
\item Fix $x_0'$, let $r_0=|x_0'|$, then the change of variables $y=(\frac{1}{r_0}x',\frac{1}{r_0^\alpha}x_n)$ changes the point $(x_0',-h(x_0'))$ on $\Gamma$ to the point $y_0=(\frac{1}{r_0}x_0',-\frac{h(x_0')}{r_0^\alpha})$ on $\partial\{v^{r_0}=0\}$. Since $-y_{0,n}\sim1$, the point $y_0$ can be dealt with standard method, for instance, the argument in \cite{1}.
\end{itemize}
\end{proof}
\section*{Acknowledgement and AI disclosure}
G. Huang is  supported in part by National Key R$\&$D Program of China 2025YFA1017600 and
National Natural Science Foundation of China under Grants 12526202.
\par ChatGPT 5.6 was only used to simplify the construction of barrier functions in Lemma \ref{Proposition 4.4.} and to check some details of calculations.

\small
\medskip\medskip
{\em Address and E-mail:}

\medskip\medskip
{\em Hanfei Han}

{\em School of Mathematical Sciences, Fudan University}

{\em 23110180010@m.fudan.edu.cn}

\medskip\medskip
{\em Genggeng Huang}

{\em School of Mathematical Sciences, Fudan University}

{\em genggenghuang@fudan.edu.cn}

\end{document}